\documentclass[amssymb, 11pt]{amsart}
\usepackage{amsmath}
\usepackage{amsthm}
\usepackage{amsfonts}
\newtheorem{thm}{Theorem}[section]
\newtheorem{theorem}[thm]{Theorem}

\newtheorem{cor}[thm]{Corollary}
\newtheorem{prop}[thm]{Proposition}

\newtheorem{lemma}[thm]{Lemma}

\newcommand{\num}{\refstepcounter{thm}}

\DeclareMathOperator\Aut{Aut}
\DeclareMathOperator\Out{Out}
\DeclareMathOperator\F{{\mathbb{F}}}
\DeclareMathOperator\Alt{Alt}

\begin{document}

\title [On fixed points of permutations] {On fixed points of permutations}

\author{Persi Diaconis}
\address{Department of Mathematics and Statistics\\
Stanford University\\ Stanford, CA 94305}

\author{Jason Fulman}
\address{Department of Mathematics\\
University of Southern California\\
Los Angeles, CA 90089-2532}
\email{fulman@usc.edu}

\author{Robert Guralnick }
\address{Department of Mathematics\\
University of Southern California\\
Los Angeles, CA 90089-2532}
\email{guralnic@usc.edu}



\thanks{Diaconis was supported by NSF grant DMS-0505673.
Fulman received funding from NSA
grant H98230-05-1-0031 and NSF grant DMS-0503901. Guralnick was
supported by  NSF grant DMS-0653873.}
\thanks{We thank Sam Payne for his elegant construction
in Section 8.4(2).}

\date{August 15, 2007}

\begin{abstract} The number of fixed points of a random permutation of $\{1, 2,
  \ldots, n\}$ has a limiting Poisson distribution. We seek a
  generalization, looking at other actions of the symmetric group.
  Restricting attention to primitive actions, a complete
  classification of the limiting distributions is given. For most
  examples, they are trivial -- almost every permutation has no fixed
  points. For the usual action of the symmetric group on $k$-sets of
  $\{1, 2, \ldots, n\}$, the limit is a polynomial in independent
  Poisson variables. This exhausts all cases. We obtain asymptotic estimates
  in some examples, and give a survey of related results. \end{abstract}

\maketitle

 {This paper is dedicated to the life and work of our
  colleague Manfred Schocker.}

\section{Introduction} \label{intro}

One of the oldest theorems in probability theory is the Montmort
(1708) limit
theorem for the number of fixed points of a random permutation of
$\{1, 2, \ldots, n\}$. Let $S_n$ be the symmetric group.
For an
element $w \in S_n$, let $A(w) = \{i : w(i) = i\}$. Montmort \cite{monmort}
proved that
\num
\begin{equation} \label{eqn0}
{|\{w : A(w) = j\}| \over n!} \rightarrow {1 \over e} \ {1 \over
  j!}
\end{equation}
for $j$ fixed as $n$ tends to infinity. The limit theorem (\ref{eqn0})  has
had many refinements and variations. See Tak\'{a}cs \cite{Ta} for its
history, Chapter 4 of Barbour, Holst, Janson \cite{BHJ} or Chatterjee,
Diaconis, Meckes \cite{CDM} for modern versions.

The limiting distribution $P_{\lambda}(j) = e^{-\lambda} \lambda^j /
j!$ (in (\ref{eqn0}) $\lambda = 1$) is the Poisson distribution of
``the law of small numbers''.  Its occurrence in many other parts of
probability (see e.g.  Aldous \cite{Al}) suggests that we seek
generalizations of (\ref{eqn0}), searching for new limit laws.

In the present paper we look at other finite sets on which $S_n$ acts.
It seems natural to restrict to transitive action -- otherwise, things
break up into orbits in a transparent way. It is also natural to
restrict to primitive actions. Here $S_n$ acts {\it primitively} on
the finite set $\Omega$ if we cannot partition $\Omega$ into disjoint
blocks $\Delta_1, \Delta_2, \ldots, \Delta_h$ where $S_n$ permutes the
blocks (if $\Delta_i^w \cap \Delta_j \not = \phi$ then
$\Delta_i^w = \Delta_j)$. The familiar wreath products which permute
within blocks and between blocks are an example of an
imprimitive action.

The primitive actions of $S_n$ have been classified in the O'Nan-Scott
theorem. We describe this carefully in Section \ref{Onan}.  For the
study of fixed points most of the cases can be handled by a marvelous
theorem of Luczak-Pyber \cite{LP}. This shows that, except for the
action of $S_n$ on $k$-sets of an $n$ set, almost all permutations
have no fixed points (we say $w$ is a derangement).  This result is
explained in Section \ref{lucpyb}. For $S_n$ acting on $k$-sets, one
can assume that $k <  n/2$, and there is a nontrivial limit if and
only if $k$ stays fixed as $n$ tends to infinity. In these cases, the
limit is shown to be an explicit polynomial in independent Poisson
random variables. This is the main content of Section
\ref{ksets}. Section \ref{matchings} works out precise asymptotics for
the distribution of fixed points in the action of $S_n$ on
matchings. Section \ref{impriv} considers more general imprimitive
subgroups. Section \ref{prim} proves that the proportion of elements
of $S_n$ which belong to a primitive subgroup not containing $A_n$ is
at most $O(n^{-2/3+\alpha})$ for any $\alpha>0$; this improves on the
bound of Luczak and Pyber \cite{LP}. Finally, Section \ref{survey}
surveys related results (including analogs of our main results for
finite classical groups) and applications of the distribution of fixed
points and derangements.

If a finite group $G$  acts on $\Omega$ with $F(w)$ the number of
fixed points of
$w$, the ``lemma that is not Burnside's'' implies that
$$
\begin{aligned}
  E(F(w)) &= \# \ \mbox{orbits of} \ G \  \mbox{on} \ \Omega \\
  E(F^2 (w)) &= \# \ \mbox{orbits of} \ G  \ \mbox{on} \ \Omega
  \times \Omega = \ \mbox{rank} := r.
\end{aligned}
$$
If $G$ is transitive on $\Omega$ with isotropy group $H$, then
the rank is also the number of orbits of $H$ on $\Omega$ and
so  equal to the number of $H-H$ double cosets in $G$.

Thus for transitive actions
\num
\begin{equation}
E(F(w)) = 1, \quad {\rm Var} (F(w)) = \ \mbox{rank}-1
\end{equation}
 In most of our examples $P(F(w) = 0) \rightarrow 1$ but because of
(1.2), this cannot be seen by moment methods. The standard second
moment method (Durrett \cite{Du}, page 16) says that a non-negative
integer random variable satisfies $P(X > E(X)/2) \geq {1 \over 4}
(EX)^2 / E(X^2)$. Specializing to our case, $P(F(w) > 0) \geq 1/(4
r)$; thus $P(F(w) = 0) \leq 1 - 1/(4 r)$. This shows that the
convergence to $1$ cannot be too rapid.

There is also a quite easy lower bound
for $P(F(w)=0)$ \cite{GW}.  Even the simplest instance of this lower
bound was only observed in 1992 in \cite{CaCo}.
We reproduce the simple proof from \cite{GW}.  Let $n=|\Omega|$
and let $G_0$ be the set of elements of $G$ with no fixed points.
Note that $F(w) \le n$, whence
$$
\sum_G (F(w)-1)(F(w)-n) \le \sum_{G_0} (F(w)-1)(F(w)-n) = n|G_0|.
$$
On the other hand, the left hand side is equal
to $|G|(r -1)$.  Thus,
$P(F(w)=0) \ge (r-1)/n$.  We record these bounds.

\begin{theorem} \label{basic}  Let $G$ be a finite
transitive permutation group
of degree $n$ and rank $r$.  Then
$$
\frac{r-1}{n} \le P(F(w)=0) \le 1 - \frac{1}{4r}.
$$
\end{theorem}

Frobenius groups of order $n(n-1)$ with $n$ a prime power
are only the possibilities when the lower bound is achieved.
The inequality above shows that $P(F(w)=0)$ tends to $1$
implies that the rank tends to infinity.  Indeed, for primitive
actions of symmetric and alternating groups, this is also
a sufficient condition -- see Theorem \ref{theC}.

\section{O'Nan-Scott Theorem} \label{Onan}

Let $G$ act transitively on the finite set $\Omega$. By standard
theory we may represent $\Omega = G/ G_{\alpha}$, with any fixed
$\alpha \in \Omega$. Here $G_{\alpha} = \{w : \alpha^w = \alpha\}$
with the action being left multiplication on the cosets. Further
(Passman \cite[3.4]{P}) the action of $G$ on $\Omega$ is primitive if and only
if the isotropy group $G_{\alpha}$ is maximal.
Thus, classifying primitive actions of $G$ is the same problem as
classifying maximal subgroups $H$ of $G$.

The O'Nan-Scott theorem classifies maximal subgroups of $A_n$ and
$S_n$ up to determining the almost simple primitive groups of degree $n$.

\begin{theorem}  {\rm [O'Nan-Scott]}  Let $H$ be a maximal
subgroup of $G=A_n$ or $S_n$. Then, one of the following three cases holds:
\begin{description}
\item [I] $H$ acts primitively as a subgroup of $S_n$ (primitive case),
\item [II] $H = (S_a \wr S_b) \cap G$ (wreath product), $n = a \cdot b, |
  \Omega | = \frac{n!}{(a!)^b \cdot b!}$ (imprimitive case), or
\item [III] $H = (S_k \times S_{n-k}) \cap G, | \Omega | = {n \choose k}$
with $1 \le k < n/2$ (intransitive case).
\end{description}
 Further, in case I, one of the following holds:
\begin{description}
\item [Ia] $H$ is almost simple,
\item [Ib] $H$ is diagonal,
\item [Ic] $H$ preserves product structure, or
\item [Id] $H$ is affine.
\end{description}
\end{theorem}

 {\it Remarks and examples:}
\begin{enumerate}

\item Note that in cases I, II, III, the modifiers
  `primitive', `imprimitive', `intransitive' apply to $H$.
  Since $H$ is maximal in $G$, $\Omega \cong G/H$ is a primitive
  $G$-set.
   We present an example and suitable additional definitions for each case.

\item In case III, $\Omega$ is the $k$-sets of $\{1, 2,
  \ldots, n\}$ with the obvious action of $S_n$. This case is
  discussed extensively in Section \ref{ksets} below.

\item In case II, take $n$ even with $a = 2, b = n/2$. We may identify
  $\Omega$ with the set of perfect matchings on $n$ points --
  partitions of $n$ into $n/2$ two-element subsets where order within
  a subset or among subsets does not matter. For example if $n = 6,
  \{1,2\} \{3,4\} \{5,6\}$ is a perfect matching. For this case,
  $|\Omega| = \frac{n!}{2^{n/2} (n/2)!} = (2 n-1)(2n-3) \cdots
  (1)$. Careful asymptotics for this case are developed in Section
  \ref{matchings}. More general imprimitive subgroups are considered in
Section \ref{impriv}.

\item While every maximal subgroup of $A_n$ or $S_n$ falls into one of the
categories of the O'Nan-Scott theorem,
not every group is maximal.   A complete list of the exceptional examples
is in
Liebeck, Praeger and Saxl \cite{LPS1}.

\item In case Ia, $H$ is {\it almost simple} if for some non-abelian
  simple group $G$, $G \leq H \leq \Aut(G)$. For example, fix $1 < k <
  m$. Let $n = {m \choose k}$. Let $S_n$ be all $n!$ permutations of
  the $k$ sets of $\{1, 2, \ldots, m\}$. Take $S_m \leq S_n$ acting on
  the $k$-sets in the usual way.  For $m \geq 5$, $S_m$ is almost
  simple and primitive. Here $\Omega = S_n / S_m$ does not have a
  simple combinatorial description, but this example is crucial and
  the $k=2$ case will be analyzed in Section \ref{prim}.

Let $\tau \in S_m$ be a transposition. Then $\tau$ moves
precisely $ 2 \binom{m-2}{k-1}$ elements of $\Omega$.
%%% just pick a k-1 set disjoint from 1,2 and adjoin 1; this is moved.
Thus, $S_m$ embeds in $A_n$ if and only if  $\binom{m-2}{k-1}$  is
even. Indeed for most primitive embeddings of $S_m$ into $S_n$, the
image is contained in $A_n$ \cite{wisc}.

It is not difficult to see that the image of $S_m$ is maximal in
either $A_n$ or $S_n$.  This follows from the general result
in \cite{LPS1}.   It also follows from the classification of primitive
groups containing a non-trivial
element fixing at least $n/2$ points \cite{GM}.

Similar examples can be constructed by
  looking at the action of $P\Gamma L_d (q)$ on $k$-spaces
  (recall the $P\Gamma L_d(q)$ is the projective group of all
  semilinear transformations of a $d$ dimensional vector space
  over $\F_q$).  All of these are
  covered by case $Ia$.

\item In case Ib, {\it $H$ is diagonal} if $H = G^k \cdot (\Out(G) \times S_k)$
for $G$ a non-abelian simple group, $k \geq 2$
  (the dot denotes semidirect product).
  Let $\Omega = G^k / D$ with $D = \{(g, g, \ldots g)\}_{g \in G}$ the
  diagonal subgroup. Clearly $G^k$ acts on $\Omega$. Let $\Out(G)$ (the
  outer automorphisms) act coordinate-wise and let $S_k$ act by
  permuting coordinates. These transformations determine a permutation
  group $H$ on the set $\Omega$. The  group $H$ has normal subgroup $G^k$
with quotient isomorphic to
  $\Out(G) \times S_k$. The extension usually splits (but it doesn't always
split).

  Here is an specific example. Take $G = A_m$ for $m \geq 8$ and $k = 2$.
Then $\Out(A_m) = C_2$ and so $H = \langle A_m \times A_m, \tau, (s,s)
\rangle$ where $s$ is a transposition (or any element in $S_m$
outside of $A_m$) and $\tau$ is the involution changing coordinates.
More precisely, each coset of $D$ has a unique representative of the
form $(1, x)$. We have $(g_1, g_2) (1, x)D =
(g_1, g_2 x)D = (1,g_2 x g_1^{-1})D$. The action
of $\tau \in C_2$ takes $(1, x) \rightarrow (1, x^{-1})$ and
the action of $(s,s) \in \Out(A_m)$ takes $(1, x)$ to $(1, sxs^{-1})$.

  The maximality of $H$ is somewhat subtle. We first show that if $m
  \geq 8$, then $H$ is contained in $\Alt(\Omega)$. Clearly $A_m \times A_m$
  is contained in $\Alt(\Omega)$. Observe that $(s,s)$ is contained in
  $\Alt(\Omega)$. Indeed, taking $s$ to be a transposition, the number of
  fixed points of $(s,s)$ is the size of its centralizer in $A_m$
  which is $|S_{m-2}|$, and so $\frac{m!}{2}-(m-2)!$ points are
  moved and this is divisible by $4$ since $m \geq 8$. To see that
  $\tau$ is contained in $Alt(\Omega)$ for $m \geq 8$, note that the number
  of fixed points of $\tau$ is the number of involutions (including
  the identity) in $A_m$, so it is sufficient to show that
  $\frac{m!}{2}$ minus this number is a multiple of 4. This follows
  from the next proposition, which is of independent combinatorial
  interest.

\begin{prop} \label{countinv} Suppose that $m \geq 8$.
Then the number of involutions in $A_m$ and the number of
involutions in $S_m$ are multiples of $4$. \end{prop}

\begin{proof} Let $a(m)$ be the number of involutions in $A_m$
(including the identity). Let $b(m)$ be the number of involutions in
$S_m-A_m$. It suffices to show that $a(m)=b(m) = 0 \mod 4$. For $n = 8,9$
we compute directly.
For $n > 9$, we observe that \[ a(n) = a(n-1) + (n-1)b(n-2) \] and \[
b(n) = b(n-1) + (n-1)a(n-2) \] (because an involution either fixes 1
giving the first term or swaps 1 with $j > 1$, giving rise to the
second term). The result follows by induction.
\end{proof}

Having verified that $H$ is contained in $\Alt(\Omega)$ for $m \geq 8$,
maximality now follows from Liebeck-Praeger-Saxl \cite{LPS1}.

\item In case Ic, {\it $H$ preserves a product structure}.
Let $\Gamma=\{1,...,m\}$, $\Delta=\{1,...,t\}$, and let $\Omega$ be
the $t$-fold Cartesian product of $\Gamma$. If $C$ is a permutation
group on $\Gamma$ and $D$ is a permutation group on $\Delta$, we may
define a group $H = C \wr D$ by having $C$ act on the coordinates,
and having $D$ permute the coordinates. Primitivity of $H$ is
equivalent to $C$ acting
  primitively on $\Gamma$ with some non identity element having a
  fixed point and $D$ acting transitively on $\Delta$ (see, e.g.
  Cameron \cite{Ca1}, Th. 4.5).

  There are many examples of case Ic but $|\Omega| = m^t$
 is rather restricted and
  $H$ has a simple form. One specific example is as follows:
$G=S_{m^t}$, $H=S_m \wr S_t$ and $\Omega$
  is the t-fold Cartesian product $\{1,\cdots,m\}^t$. The case
$t=2$ will be analyzed in detail in Section \ref{prim}.
  It is easy to determine when $H$ embeds in $A_{m^t}$.  We just note that
  if $t=2$, then this is case if and only if $4|m$.

\item In case Id $H$ is affine. Thus $\Omega = V$, a vector
  space of dimension $k$ over a field of $q$ elements (so $n= |\Omega | =
  q^k$) and $H$ is the semidirect product $V \cdot GL(V)$.    Since
  we are interested only in maximal subgroups, $q$ must be prime.

 Note that if $q$ is odd, then $H$ contains an $n-1$ cycle and so is
 not contained in $A_n$.  If $q=2$, then for $k > 2$, $H$ is perfect
 and so is contained in $A_n$.   The maximality of $H$ in $A_n$
 or $S_n$ follows by
Mortimer \cite{mortimer} for $k >1$ and \cite{gurkim} if $k=1$.

\item  The proof of the
O'Nan-Scott theorem  is not extremely difficult.
 O'Nan and
Scott each presented proofs at the Santa Cruz Conference in 1979.
 There is a more delicate version which describes
all primitive permutation groups.   This was  proved in
 Aschbacher-Scott \cite{AS} giving quite detailed information.
A short proof of the Aschbacher-O'Nan Scott Theorem is in \cite{msri}.
See also  Liebeck, Praeger
and Saxl \cite{LPS2}). A textbook presentation is in Dixon and
Mortimer \cite{DxM}. We find the lively lecture notes of Cameron
(\cite{Ca1}, Chapter 4) very helpful. The theorem has a life of its
own, away from permutation groups, in the language of the generalized
Fitting subgroup $F^*$. See Kurtzweil and Stellmacher \cite{KS}.

 See also the  lively lecture notes of Cameron
(\cite{Ca1}, Chapter 4).    The notion of generalized Fitting
subgroup is quite useful in both the proof and statement
of the theorem.      See Kurtzweil and Stellmacher \cite{KS}.

\item It turns out that many of the details above are not needed for
  our main results. Only case III $(H = S_k \times S_{n-k})$ allows
  non-trivial limit theorems. This is the subject of the next
  section. The other cases are of interest when we try to get explicit
  bounds (Sections \ref{matchings}, \ref{impriv}, \ref{prim}).

\end{enumerate}

\section{Two Theorems of Luczak-Pyber} \label{lucpyb}

The following two results are due to Luczak and Pyber.

\begin{theorem} \label{theA} (\cite{LP}) Let $S_n$ act on $\{1,
2, \ldots, n\}$ as usual and let $i(n, k)$ be the number of $w \in
S_n$ that leave some $k$-element set invariant. Then, ${i (n, k) \over
  n!}  \leq a k^{-.01}$ for an absolute constant $a$.
 \end{theorem}

\begin{theorem} \label{theB} (\cite{LP}) Let $t_n$ denote the
number of elements of the symmetric group $S_n$ which belong to
transitive subgroups different from $S_n$ or $A_n$. Then
$$
\lim_{n \rightarrow \infty} t_n / n! = 0.
$$ \end{theorem}

Theorem \ref{theA} is at the heart of the proof of Theorem \ref{theB}.
We use them both
to show that a primitive action of $S_n$ is a derangement with
probability approaching one, unless $S_n$ acts on $k$-sets with fixed
$k$. Note that we assume that $k \leq n/2$ since the action on $k$-sets is
isomorphic to the action on $n-k$ sets.

\begin{theorem} \label{theC} Let $G_i$ be a finite symmetric
or alternating group of degree $n_i$    acting
primitively on a finite set $\Omega_i$ of cardinality at
least $3$.  Assume
that $n_i \rightarrow \infty$.   Let
$d_i$ be the proportion of $w \in G_i$ with no fixed points. Then
the following are equivalent:
\begin{enumerate}
\item
$
\lim_{i \rightarrow \infty} d_i  = 1,
$
\item there is no fixed $k$  with $\Omega_i= \{k-\mbox{sets of} \
\{1, 2, \ldots, n_i\}\}$ for infinitely many $i$, and
\item the rank of $G_i$ acting on $\Omega_i$ tends to $\infty$.
\end{enumerate}
\end{theorem}

\begin{proof}  Let $H_i$ be an isotropy group for $G_i$
acting on $\Omega_i$.   If $H_i$
falls into category I or II of the O'Nan-Scott theorem,
$H_i$ is transitive. Writing out Theorem \ref{theB} above more fully,
Luczak-Pyber prove that
$$
{\left| \bigcup_H  H \right| \over n!} \rightarrow 0
$$ where the union is over {\it all} transitive subgroups of $S_n$
not equal to $S_n$ or $A_n$. Thus a randomly chosen $w \in S_n$ is not
in $x H_n x^{-1}$ for any $x$ if $H_n$ falls into category I or II.

Having ruled out categories I and II, we turn to category III
($k$-sets of an $n$ set). Here, Theorem \ref{theA} shows the chance of
a derangement tends to one as $a/k^{.01}$, for an absolute constant
$a$.

The previous paragraphs show that (2) implies (1). If the rank does not
go to $\infty$, then $d_i$ cannot approach 1 by Theorem
\ref{basic}. Thus (1) implies (3), and also (2) since the rank of the
action on k-sets is $k+1$. Clearly (3) implies (2), completing the proof.
\end{proof}

\section{$k$-Sets of an $n$-Set} \label{ksets}

In this section the limiting distribution of the number of fixed
points of a random permutation acting on $k$-sets of an $n$-set is
determined.

\begin{theorem} \label{klim} Fix $k$ and let $S_n$ act on $\Omega_{n,
  k}$ -- the $k$ sets of $\{1, 2, \ldots, n\}$. Let $A_i (w)$ be the
number of $i$-cycles of $w \in S_n$ in its usual action on $\{1, 2,
\ldots, n\}$. Let $F_k(w)$ be the number of fixed points of $w$ acting
on $\Omega_{n, k}$. Then
\num
\begin{equation}
F_k(w) = \sum_{|\lambda|=k} \prod_{i=1}^k {A_i (w) \choose
  \alpha_i (\lambda)}.
\end{equation}
 Here the sum is over partitions $\lambda$ of $k$ and $\alpha_i
(\lambda)$ is the number of parts of $\lambda$ equal to $i$.

 (2) For all $n \geq 2, E(F_k) = 1, {\rm Var} (F_k) = k$.

 (3) As $n$ tends to infinity, $A_i(w)$ converge to independent
Poisson $(1/i)$ random variables. \end{theorem}

\begin{proof} If $w \in S_n$ is to fix a $k$ set, the cycles
  of $w$ must be grouped to partition $k$. The expression for $F_k$
  just counts the distinct ways to do this. See the examples below.
  This proves (1). The rank of $S_n$ acting on $k$ sets is $k+1$,
  proving (2).

  The joint limiting distribution of the $A_i$ is a classical result
due to Goncharov \cite{Go}. In fact, letting $X_1,X_2,\cdots,X_k$ be
independent Poisson with parameters $1,\frac{1}{2},\cdots,\frac{1}{k}$,
one has from \cite{DS} that for all $n \geq \sum_{i=1}^k ib_i$,
\[ E \left( \prod_{i=1}^k A_i(w)^{b_i} \right) =
\prod_{i=1}^k E(X_i^{b_i}).\] For total variation bounds see \cite{AT}.
This proves (3). \end{proof}

 {\it Examples}. Throughout, let $X_1, X_2, \ldots, X_k$ be
independent Poisson random variables with parameters $1, 1/2, 1/3,
\ldots, 1/k$ respectively.

{$k=1$:} This is the usual action of $S_n$ on $\{1, 2, \ldots,
n\}$ and Theorem \ref{klim} yields (1) of the introduction: In
particular, for derangements
$$
P(F_1 (w) = 0) \rightarrow 1/e \ \dot = \ .36788 .
$$

{$k=2$:} Here $F_2(w) = {A_1 (w) \choose 2} +
A_2 (w)$ and Theorem \ref{klim} says that

$P(F_2(w) = j) \rightarrow P \left( {X_1 \choose 2} +
X_2 = j \right)$ with $X_1$ Poisson$(1)$, $X_2$ Poisson$({1 \over 2})$.

In particular
$$
P(F_2(w) = 0) \rightarrow {2 \over e^{3/2}} \ \dot = \ .44626 .
$$

{$k=3$:} Here $F_3(w) = { A_1 (w) \choose 3}
+ A_1 (w) A_2 (w) + A_3 (w)$ and
$$ P(F_3 (w) = j) \rightarrow P\left( {X_1 \choose 3} + X_1 X_2 + X_3
= j \right) .
$$

In particular
$$
P(F_3(w) = 0) \rightarrow {1 \over e^{4/3}} (1 + {3 \over 2}
e^{-1/2}) \ \dot = \ .50342.
$$

We make the following conjecture, which has also been independently stated
as a problem by Cameron \cite{Ca2}. \\

\noindent
 {\it Conjecture:} $\lim_{n \rightarrow \infty} P(F_k (w) = 0)$ is increasing
in
$k$. \\

Using Theorem \ref{klim}, one can prove the following result
which improves, in this context, the upper bound given in Theorem \ref{basic}.

\begin{prop} \[ lim_{n \rightarrow \infty} P(F_k (w) = 0)
\leq 1 - \frac{\log(k)}{k} + O \left( \frac{1}{k} \right). \]
\end{prop}

\begin{proof} Clearly
\begin{eqnarray*} P(F_k(w) > 0) & \geq & P \left(\bigcup_{j=1}^{\lfloor
\frac{k-1}{2} \rfloor} (A_{k-j} > 0 \ \mbox{and} \ A_j > 0) \right)\\
& = & 1 - P \left( \bigcap_{j=1}^{\lfloor \frac{k-1}{2} \rfloor}
\overline{ (A_{k-j} > 0 \ \mbox{and} \ A_j > 0) }
\right). \end{eqnarray*} By Theorem \ref{klim}, this converges to \[ 1
- \prod_{j=1}^{\lfloor \frac{k-1}{2} \rfloor} \left[ 1-(1- e^{-{1/j}})
(1 - e^{-{1/(k-j)}}) \right]. \] Let $a_j=(1-e^{-1/j})$. Write the
general term in the product as $e^{\log(1-a_ja_{k-j})}$. Expand the
log to $-a_ja_{k-j} + O \left( (a_ja_{k-j})^2 \right)$. Writing $a_j=
\left( \frac{1}{j}+O(\frac{1}{j^2}) \right)$ and multiplying out, we
must sum \[ \sum_{j=1}^{\lfloor \frac{k-1}{2} \rfloor} \frac{1}{j}
\frac{1}{k-j} , \sum_{j=1}^{\lfloor \frac{k-1}{2} \rfloor}
\frac{1}{j^2} \frac{1}{k-j} , \sum_{j=1}^{\lfloor \frac{k-1}{2}
\rfloor} \frac{1}{(k-j)^2} \frac{1}{k} , \sum_{j=1}^{\lfloor
\frac{k-1}{2} \rfloor} \frac{1}{(k-j)^2} \frac{1}{j^2}.\] Writing
$\frac{1}{j} \frac{1}{k-j} = \frac{1}{k} \left( \frac{1}{j} +
\frac{1}{k-j} \right),$ the first sum is $\frac{\log(k)}{k}+ O \left(
\frac{1}{k} \right)$. The second sum is $O \left( \frac{1}{k}
\right)$, the third sum is $O \left( \frac{\log(k)}{k^2} \right)$ and
the fourth is $O \left( \frac{1}{k^2} \right)$. Thus $-a_j a_{k-j}$
summed over $1 \leq j \leq (k-1)/2$ is $- \frac{\log(k)}{k} + O \left(
\frac{1}{k} \right)$. The sum of $(a_ja_{k-j})^2$ is of lower order by
similar arguments. In all, the lower bound on $lim_{n \rightarrow
\infty} P(F_k(w)>0)$ is \[ 1-e^{-\frac{\log(k)}{k} + O \left(
\frac{1}{k} \right)} = \frac{\log(k)}{k} + O \left( \frac{1}{k}
\right).\] \end{proof}

To close this section, we give a combinatorial interpretation for
the moments of the numbers $F_k(w)$ of Theorem \ref{klim} above.
This involves
the ``top k to random'' shuffle, which removes k cards from the
top of the deck, and randomly
interleaves them with the other n-k cards (choosing one of the ${n
\choose k}$ possible interleavings uniformly at random).

\begin{prop} \label{shufeig}
\begin{enumerate}
\item The eigenvalues of the top k to random shuffle are the
numbers $\left\{ \frac{F_k(w)}{{n \choose k}} \right \}$,
where $w$ ranges over $S_n$.

\item For all values of $n,k,r$, the rth moment of the distribution of
fixed k-sets is equal to ${n \choose k}^r$ multiplied by the chance
that the top k to random shuffle is at the identity after r steps.
\end{enumerate}
\end{prop}

\begin{proof} Note that the top k to random shuffle is the
inverse of the move k to front shuffle, which picks k cards at
random and moves them to the front of
the deck, preserving their relative order. Hence their transition
matrices are transposes, so have the same eigenvalues. The move k to
front shuffle is a special case of the theory of random walk on
chambers of hyperplane arrangements developed in \cite{BHR}. The
arrangement is the braid arrangement and one assigns weight
$\frac{1}{{n \choose k}}$ to each of the block ordered partitions
where the first block has size $k$ and the second block has size
$n-k$. The result now follows from Corollary 2.2 of \cite{BHR}, which
determined the eigenvalues of such hyperplane walks.

For the second assertion, let $M$ be the transition matrix for the top k to
random shuffle.
Clearly $Tr(M^r)$ (the trace of $M^r$) is equal to $n!$ multiplied
by the chance that the top k to
random shuffle is at the identity after r steps. The first part
gives that \[ Tr(M^r) = \sum_{\
w \in S_n} \left( \frac{F_k(w)}{{n \choose k}} \right)^r, \]
which implies the result. \end{proof}

As an example of part 2 of Proposition \ref{shufeig}, the chance of
being at the identity after 1 step is $\frac{1}{{n \choose k}}$ and
the chance of being at the identity after 2 steps is $\frac{k+1}{{n
\choose k}^2}$, giving another proof that $E(F_k(w))=1$ and
$E(F_k^2(w))=k+1$.

{\it Remarks}
\begin{enumerate}
\item As in the proof of Proposition \ref{klim}, the moments of $F_k(w)$
can be expressed exactly in terms of the moments of Poisson random
variables, provided that $n$ is sufficiently large.

\item There is a random walk on the irreducible representations of $S_n$
which has the same eigenvalues as the top k to random walk, but with
different multiplicities. Unlike the top k to random walk, this walk is
reversible with respect to its stationary distribution, so that spectral
techniques (and hence information about the distribution of fixed points)
can be used to analyze its convergence rate. For details, applications,
and a generalization to other actions, see \cite{F1}, \cite{F2}.
\end{enumerate}

\section{Fixed Points on Matchings} \label{matchings}

Let $M_{2n}$ be the set of perfect matchings on $2n$ points. Thus, if
$2n = 4, M_{2n} = \{(1,2)(3,4), (1,3)(2,4), (1,4)(2,3)\}$. It is well
known that
$$
|M_{2n}| = (2n - 1)!! = (2n - 1)(2n - 3) \cdots (3) (1).
$$

 The literature on perfect matchings is enormous. See Lov\'{a}sz and
Plummer \cite{LoPl} for a book length treatment. Connections with
phylogenetic trees and further references are in \cite{DH,DH2}. As
explained above, the symmetric group $S_{2n}$ acts primitively on
$M_{2n}$. Results of Luczak-Pyber \cite{LP} imply that, in this action,
almost every permutation is a derangement. In this section we give
sharp asymptotic rates for this last result. We show that the
proportion of derangements in $S_{2n}$ is
\num
\begin{equation} \label{eqn5.1}
1 - {A (1) \over \sqrt{\pi n}} + o \Bigg( {1 \over \sqrt{n}} \Bigg) ,
\quad A(1) = \prod_{i=1}^\infty {\rm cosh} (1 / (2i-1)).
\end{equation}

 Similar asymptotics are given for the proportion of permutations
with $j > 0$ fixed points. This is zero if $j$ is even. For odd $j$,
it is
\num
\begin{equation} \label{eqn5.2}
 {C(j) B(1) \over \sqrt{\pi n}} + o \Bigg( {1 \over \sqrt{n}} \Bigg), \quad
B(1) = \prod_{i=1}^\infty \Bigg( 1 + {1 \over 2i} \Bigg) e^{-{1 / 2i}}
\end{equation}
 and $C(j)$ explicit rational numbers. In particular
\num
\begin{equation} \label{eqn5.3}
C(1) = {3 \over 2}, \ C(3) = {1 \over 4}, \ C(5) = {27 \over 400}, \
C(7) = {127 \over 2352}.
\end{equation}

The argument proceeds by finding explicit closed forms for
generating functions followed by standard asymptotics. It is well
known that the rank of this action is $p(n)$, the number of
partitions of $n$.  Thus (\ref{eqn5.1}) is a big improvement over
the upper bound given in Theorem \ref{basic}.

For $w \in S_{2n}$, let $a_i (w)$ be the number of $i$-cycles in the
cycle decomposition. Let $F(w)$ be the
number of fixed points of $w$ acting on $M_{2n}$. The following
proposition determines $F(w)$ in terms of $a_i(w), \ 1 \leq i \leq
2n$.

\begin{prop} \label{numfix} The number of fixed points, $F(w)$ of $w
\in S_{2n}$ on $M_{2n}$ is
$$
F(w) = \prod_{i=1}^{2n} F_i (a_i (w))
$$
 with
$$
F_{2i-1} (a) = \begin{cases} 1 &{\rm if} \ a = 0 \\ 0 &{\rm if} \ a
  \ \mbox{is odd} \\ (a-1)!!(2i-1)^{a/2} &{\rm if} \ a > 0 \ \mbox{is
    even} \end{cases}
$$
$$
F_{2i} (a) = 1+ \sum_{k=1}^{\lfloor a/2 \rfloor} (2k-1)!! {a \choose 2k}
  (2i)^k $$ In particular,
\num
\begin{equation}
F(w) \not = 0 \ \mbox{if and only if} \ a_{2i-1} (w) \ \mbox{is
  even for all} \ i
\end{equation}
\num
\begin{equation}
F(w) \ \mbox{does not take on non-zero even values}.
\end{equation}
\end{prop}

\begin{proof} Consider first the cycles of $w$ of length $2
i-1$. If $a_{2i-1}$ is even, the cycles may be matched in pairs, then
each pair of length $2i-1$ can be broken into  matched two
element subsets by first pairing the lowest element among the numbers
with any of the $2i-1$ numbers in the second cycle. The rest is
determined by cyclic action. For example, if the two three-cycles
$(123) (456)$ appear, the matched pairs $(14)(25)(36)$ are fixed, so
are $(15)(26)(34)$ or $(16)(24)(35)$. Thus $F_3 (2) = 3$. If
$a_{2i-1}$ is odd, some element cannot be matched and $F_{2i-1}(a_{2i-1})
= 0$.

Consider next the cycles of $w$ of length $2i$. Now, there are two
ways to create parts of a fixed perfect matching. First, some of these
cycles can be paired and, for each pair, the previous construction can
be used. Second, for each unpaired cycle, elements $i$ apart can be
paired. For example, from $(1234)$ the pairing $(13)(24)$ may be
formed. The sum in $F_{2i} (a)$ simply enumerates by partial
matchings.

To see that $F(w)$ cannot take on non-zero even values, observe that
$F_{2i-1} (a)$ and $F_{2i}(a)$ only take on odd values if they are
non-zero. \end{proof}

Let $P_{2n} (j) = \frac{|\{w \in S_{2n} : F(w) = j\}|}{2n!}$.
For $j \geq 1$, let $g_j (t) = \sum_{n=0}^\infty
  t^{2n} P_{2n} (j)$ and let $\bar g_0(t) = \sum_{n=0}^\infty t^{2n} (1 -
  P_{2n} (0))$.

\begin{prop} \label{cycindex}
$$
\bar g_0(t) = {\displaystyle \prod_{i=1}^\infty
  \cosh (t^{2i-1} / (2i-1)) \over \sqrt {1 - t^2}} \qquad
\mbox{for} \ 0 < t < 1.
$$ \end{prop}

\begin{proof} From Proposition \ref{numfix}, $w \in S_{2n}$ has $F(w)
\not = 0$ if and only if $a_{2i-1} (w)$ is even for all $i$. From
Shepp-Lloyd \cite{SL}, if $N$ is chosen in $\{0, 1, 2, \ldots\}$ with
$$
P (N = n) = (1 - t) t^n
$$ and then $w$ is chosen uniformly in $S_N$, the $a_i(w)$ are
independent Poisson random variables with parameter $t^i / i$
respectively. If $X$ is a Poisson $(\lambda)$ random variable, $P$ ($X$
is even) = $\Bigg( {1 \over 2} + {e^{-2 \lambda} \over 2} \Bigg)$. It
follows that
\begin{eqnarray*}
  \sum_{n=0}^\infty (1 - t) t^{2n} (1 - P_{2n} (0)) & = &
  \prod_{i=1}^\infty \Bigg( {1 + e^{-2{t}^{2i-1} / (2i-1)} \over 2}
  \Bigg)\\ & = & \prod_{i=1}^\infty e^{{-t}^{2i-1} / (2i-1)}
  \prod_{i=1}^\infty {\rm
    cosh} \Bigg( {t^{2i-1} \over (2i-1)} \Bigg) \\
  & = & \sqrt{{1 - t \over 1 + t}} \prod_{i=1}^\infty {\rm cosh}
  (t^{2i-1} / (2i-1)). \end{eqnarray*} \end{proof}

\begin{cor} \label{asymcor}
 As $n$ tends to infinity,
$$
1 - P_{2n} (0) \sim {\displaystyle \prod_{i=1}^\infty {\rm cosh}
  (1/(2i-1)) \over \sqrt {\pi n}}.
$$ \end{cor}

\begin{proof} By Proposition \ref{cycindex}, $1-P_{2n}(0)$ is the
coefficient of $t^n$ in \[ \frac{\prod_{i=1}^\infty
  \cosh (t^{(2i-1)/2} / (2i-1))}{ \sqrt {1 - t}}.\] It is
  straightforward to check that the numerator is analytic near $t=1$,
  so the result follows from Darboux's theorem (\cite{O}, Theorem
  11.7). \end{proof}

Proposition \ref{numfix} implies that the event $F(w) = j$ is
contained in the event $\{a_{2i-1} (w)$ is even for all $i$ and
$a_{2i}(w) \in \{0, 1\} \ \mbox{for} \ 2i > j\}$. This is evidently
complicated for large $j$.

We prove

\begin{prop} \label{genfunc} For positive odd $j$,
$$g_j (t) = {P_j(t) \displaystyle \prod_{i=1}^{\infty} \Big(1 +
{t^{2i} \over 2i} \Big)
  e^{-t^{2i} / 2i} \over \sqrt {1 - t^2}}
$$
for $P_j(t)$ an explicit
rational function in $t$ with positive rational coefficients. In
particular,
$$
P_1 (t) = \Bigg(1 + {t^2 \over 2} \Bigg), P_3 (t) = \Bigg( \frac{t^4}{6} +
\frac{t^6}{18} + \frac{t^8}{36} \Bigg)
$$
$$P_5(t) =
{\frac{1}{2} \Big({1 + \frac{t^2}{2}} \Big) \Big( {t^2 \over 4} \Big)^2
\over 1 +
  {t^{4} \over 4}} + \frac{1}{2} \left(1+ \frac{t^2}{2} \right)
\left( \frac{t^5}{5} \right)^2
$$
$$
P_7 (t) = { \frac{1}{2} \Big( 1 + \frac{t^2}{2} \Big) \over 1 + {t^6 \over 6}}
\Big({t^6 \over 6} \Big)^2 + \frac{1}{6} \Big( {t^2 \over 2} \Big)^3 +
\frac{1}{2} \left(1+\frac{t^2}{2} \right) \left( \frac{t^7}{7} \right)^2.
$$ \end{prop}

\begin{proof} Consider first the case of $j=1$. From
Proposition \ref{numfix}, $F(w) = 1$ if and only if $a_{2i-1} (w) = 0$ for $i
\geq 2$, $a_1 (w) \in \{0, 2\}$ and $a_{2i} (w) \in \{0, 1\}$ for all
$i \geq 1$. For example, if $2n = 10$ and $w = (1) (2) (3456789 10)$,
the unique fixed matching is $(1\ 2) (3 \ 7) (4\ 8)(5\ 9)(6\ 10)$. From the
cycle index argument used in Proposition \ref{cycindex},
\begin{eqnarray*}
& &  \sum_{n=0}^{\infty} (1 - t) t^{2n} P_{2n} (F(w) = 1)\\
 & = & e^{-t} \Big( 1 + {t^2 \over
    2} \Big) \displaystyle \prod_{i=2}^\infty e^{-t^{2i-1} / (2i-1)}
  \prod_{i=1}^{\infty} e^{-t^{2i}/2i} \Big( 1 + {t^{2i} \over 2i} \Big) \\
  & = & (1 - t) \Big(1 + {t^2 \over 2} \Big) \displaystyle
  \prod_{i=1}^\infty \Big( 1 + {t^{2i} \over 2i} \Big)\\ & = & {(1-t)\Big( 1 +
    {t^2 \over 2} \Big) \displaystyle \prod_{i=1}^\infty \Big( 1 +
    {t^{2i} \over 2i} \Big) e^{-t^{2i} / 2i} \over \sqrt {1 - t^2}}.
\end{eqnarray*}

 The arguments for the other parts are similar. In particular, $F(w)=3$
iff one of the following holds: \begin{itemize}
 \item $a_1(w) = 4$, all $a_{2i-1}(w) = 0 \ i \geq
  2$, all $a_{2i}(w) \in \{0, 1\}$
\item $a_1(w) \in
  \{0, 2\}, \ a_2(w) = 2$ all $a_{2i-1}(w) = 0$ and
  $a_{2i} (w) \in \{0, 1\} \ i \geq 2$
\item $a_1(w) \in \{0,
  2\}, \ a_3(w) = 2, a_{2i-1} (w) = 0 \ i \geq 3, \ a_{2i} \in \{0,
  1\}$ \end{itemize} Similarly, $F(w)=5$ iff one of the following holds:
\begin{itemize}
\item
$a_4(w) = 2, \ a_1 (w) \in \{0, 2\},
a_{2i-1}(w) = 0, a_{2i} (w) \in \{0, 1\}$ else
\item $a_5(w)=2, a_1(w) \in \{0,2\}, a_{2i-1}(w)=0$, $a_{2i}(w)
\in \{0,1\}$ else
\end{itemize} Finally, $F(w)=7$ iff one of the following holds:
\begin{itemize}
\item
 $a_1(w) \in \{0, 2\}, \ a_6 (w) = 2 \
\mbox{or} \ a_2 (w) = 3, a_{2i-1} (w) = 0, a_{2i} (w)
%%% changed {0,1} to {0.2} below per Jason
\in \{0, 1\}$ else
\item $a_7(w)=2,a_1(w) \in
\{0,2\}, a_{2i-1}(w)=0$, $a_{2i}(w) \in \{0,1 \}$ else
\end{itemize} Further details are omitted. \end{proof}

The asymptotics in (\ref{eqn5.2}) follow from Proposition \ref{genfunc},
by the same method used to prove (\ref{eqn5.1}) in Corollary \ref{asymcor}.

\section{More imprimitive subgroups} \label{impriv}

    Section \ref{matchings} studied fixed points on matchings, or
    equivalently fixed points of $S_{2n}$ on the left cosets of
    $S_2 \wr S_n$. This section uses a quite different approach to
study derangements of $S_{an}$ on the left cosets of $S_a \wr
S_n$, where $a \geq 2$ is constant. It is proved that the
proportion of elements of $S_{an}$ which fix at least one left
coset of $S_a \wr S_n$ (or equivalently are conjugate to an
element of $S_a \wr S_n$ or equivalently fix a system of $n$ blocks of
size $a$) is
at most the coefficient of $u^n$
    in \[ \exp \left( \sum_{k \geq 1} \frac{u^k}{a!}
    (\frac{1}{k})(\frac{1}{k}+1) \cdots (\frac{1}{k}+a-1) \right),
    \] and that this coefficient is asymptotic to $C_a
    n^{\frac{1}{a}-1}$ as $n \rightarrow \infty$, where $C_a$ is
    an explicit constant depending on $a$ (defined in Theorem
    \ref{genfunction} below). In the special case of matchings
    ($a=2$), this becomes $\frac{ e^{\frac{\pi^2}{12}} }
    {\sqrt{\pi n}}$, which is extremely close to the true
    asymptotics obtained in Section \ref{matchings}. Moreover,
this generating function will be crucially applied when we sharpen a
result of Luczak and Pyber in Section \ref{prim}.

    The method of proof is straightforward. Clearly the number of
    permutations in $S_{an}$ conjugate to an element of $S_a \wr
    S_n$ is upper bounded by the sum over conjugacy classes $C$ of
    $S_a \wr S_n$ of the size of the $S_{an}$ conjugacy class of
    $C$. Unfortunately this upper bound is hard to compute, but we
    show it to be smaller than something which can be exactly
    computed as a coefficient in a generating function. This will
    prove the result.

    From Section 4.2 of \cite{JK}, there is the following useful
    description of conjugacy classes of $G \wr S_n$ where $G$ is a
    finite group. The classes correspond to matrices $M$ with
    natural number entries $M_{i,k}$, rows indexed by the
    conjugacy classes of $G$, columns indexed by the numbers
    $1,2,\cdots,n$, and satisfying the condition that $\sum_{i,k}
    k M_{i,k} = n$. More precisely, given an element
    $(g_1,\cdots,g_n; \pi)$ in $G \wr S_n$, for each k-cycle of
    $\pi$ one multiplies the $g$'s whose subscripts are the
    elements of the cycle in the order specified by the
    cycle. Taking the conjugacy class in $G$ of the resulting
    product contributes 1 to the matrix entry whose row
    corresponds to this conjugacy class in $G$ and whose column is
    $k$.

    The remainder of this section specializes to $G=S_a$. Since
    conjugacy classes of $S_a$ correspond to partitions $\lambda$
    of $a$, the matrix entries are denoted by $M_{\lambda,k}$. We
    write $|\lambda|= a$ if $\lambda$ is a partition of $a$. Given
    a partition $\lambda$, let $n_i(\lambda)$ denote the number of
    parts of size $i$ of $\lambda$.

\begin{prop} \label{param}  Let
 the conjugacy class $C$ of $S_a \wr S_n$ correspond to the matrix
 $(M_{\lambda,k})$ where $\lambda$ is a partition of $a$. Then the
 proportion of elements of $S_{an}$ conjugate to an element of $C$ is
 at most \[ \frac{1}{\prod_{k} \prod_{|\lambda|= a} M_{\lambda,k}! [
 \prod_i (ik)^{n_i(\lambda)} n_i(\lambda)!]^{M_{\lambda,k}}}.\]
 \end{prop}

\begin{proof} Observe that the number of cycles of length $j$ of an
element of $C$ is equal
to \[ \sum_{k|j} \sum_{|\lambda|=a} M_{\lambda,k} n_{j/k}(\lambda).\]
To see this, note
that $S_a \wr S_n$ can be viewed concretely as a permutation of $an$
symbols by letting it
act on an array of $n$ rows of length $a$, with $S_a$ permuting within
each row and $S_n$ permuting among the rows.

    Hence by a well known formula for conjugacy class sizes in a
 symmetric group, the proportion of elements of $S_{an}$ conjugate to an
 element of $C$ is equal to
\begin{eqnarray*}
& & \frac{1}{\prod_j j^{\sum_{k|j} \sum_{|\lambda|= a} M_{\lambda,k}
n_{j/k}(\lambda)} [\sum_{k|j} \sum_{|\lambda|= a} M_{\lambda,k}
n_{j/k}(\lambda)] !}\\ & \leq & \frac{1}{\prod_j j^{\sum_{k|j}
\sum_{|\lambda|= a} M_{\lambda,k} n_{j/k}(\lambda)} \prod_{k|j}
\prod_{|\lambda|= a} M_{\lambda,k} n_{j/k}(\lambda) !}\\ & \leq &
\frac{1}{\prod_j j^{\sum_{k|j} \sum_{|\lambda|= a} M_{\lambda,k}
n_{j/k}(\lambda)} \prod_{k|j} \prod_{|\lambda|= a} [ M_{\lambda,k}!
n_{j/k}(\lambda)!^{M_{\lambda,k}}]}\\ & = & \frac{1}{\prod_k
\prod_{|\lambda|= a} M_{\lambda,k}! [\prod_i (ik)^{n_i(\lambda)}
n_i(\lambda)!]^{M_{\lambda,k}}},
\end{eqnarray*} as desired. The first inequality uses the fact
that $(x_1+\cdots+x_n)! \geq x_1! \cdots x_n!$. The second
inequality uses that $(xy)! \geq x! y!^x$ for $x,y \geq 1$ integers,
which is true since \[ (xy)! = \prod_{i=1}^x \prod_{j=0}^{y-1} (i+jx)
\geq \prod_{i=1}^x \prod_{j=0}^{y-1} i(1+j) = (x!)^y (y!)^x \geq x! (y!)^x.\]
The final equality used the change of variables $i=j/k$. \end{proof}

    To proceed further, the next lemma is useful.

\begin{lemma} \label{numcycles} \[ \sum_{|\lambda|= a}
\frac{1}{\prod_i (ik)^{n_i(\lambda)} n_i(\lambda)!} = \frac{(\frac{1}{k})
(\frac{1}{k}+1) \cdots (\frac{1}{k}+a-1)}{a!} .\] \end{lemma}

\begin{proof} Let $c(\pi)$ denote the number of cycles of a permutation $\pi$.
Since the number of permutations in $S_a$ with $n_i$ cycles of length $i$
is $\frac{a!}{\prod_i i^{n_i} n_i!}$, the left hand side is equal to
\[ \frac{1}{a!} \sum_{\pi \in S_a} k^{-c(\pi)} .\] It is well known and
easily proved by induction that \[ \sum_{\pi \in S_a} x^{c(\pi)} =
x(x+1) \cdots (x+a-1).\] \end{proof}

    Theorem \ref{genfunction} applies the preceding results
    to obtain a useful generating function.

\begin{theorem} \label{genfunction}
\begin{enumerate}
\item The proportion of elements in $S_{an}$ conjugate to an element
of $S_a \wr S_n$ is at most the coefficient of $u^n$ in \[ \exp \left(\sum_{k
\geq 1} \frac{u^k}{a!} (\frac{1}{k})(\frac{1}{k}+1) \cdots
(\frac{1}{k}+a-1) \right).\]
\item For $a$ fixed and $n \rightarrow \infty$, the coefficient of
$u^n$ in this generating function is asymptotic to \[
\frac{e^{\sum_{r=2}^a p(a,r) \zeta(r)}}{\Gamma(1/a)}
n^{\frac{1}{a}-1}\] where $p(a,r)$ is the proportion of permutations
in $S_a$ with exactly r cycles, $\zeta$ is the Riemann zeta function,
and $\Gamma$ is the gamma function.
\end{enumerate}
\end{theorem}

\begin{proof} Proposition \ref{param} implies that the sought
proportion is at most the coefficient of $u^n$ in \begin{eqnarray*}
 & & \prod_{k} \prod_{|\lambda|=a} \sum_{M_{\lambda,k} \geq 0} \frac{u^{k
M_{\lambda,k}}}{ M_{\lambda,k}! [(ik)^{n_i(\lambda)}
n_i(\lambda)!]^{M_{\lambda,k}}}\\ & = & \prod_k
\prod_{|\lambda|=a} \exp \left(\frac{u^k}{(ik)^{n_i(\lambda)}
n_i(\lambda)!} \right)\\ & = & \prod_k \exp
\left(\sum_{|\lambda|=a} \frac{u^k}{(ik)^{n_i(\lambda)} n_i(\lambda)!}
\right)\\ & = & \exp \left(\sum_{k \geq 1} \frac{u^k}{a!}
(\frac{1}{k})(\frac{1}{k}+1) \cdots (\frac{1}{k}+a-1) \right).
\end{eqnarray*} The last equality used Lemma \ref{numcycles}.

    For the second assertion, one uses Darboux's lemma (see
    \cite{O} for an exposition), which gives the asymptotics of
    functions of the form $(1-u)^{\alpha} g(u)$ where $g(u)$ is
    analytic near 1, $g(1) \neq 0$, and $\alpha \not \in \{0,1,2,
    \cdots \}$. More precisely it gives that the coefficient of
    $u^n$ in $(1-u)^{\alpha} g(u)$ is asymptotic to
    $\frac{g(1)}{\Gamma(-\alpha)} n^{-\alpha-1}$. By Lemma
    \ref{numcycles}, \begin{eqnarray*} & & \exp \left( \sum_{k
    \geq 1} \frac{u^k}{a!} (\frac{1}{k})(\frac{1}{k}+1) \cdots
    (\frac{1}{k}+a-1) \right)\\ & = & \exp \left( \sum_{k \geq 1}
    \frac{u^k}{ak} + \sum_{k \geq 1} u^k \sum_{r=2}^a p(a,r)
    k^{-r} \right)\\ & = & (1-u)^{-\frac{1}{a}} \cdot \exp
    \left(\sum_{r=2}^a p(a,r) \sum_{k \geq 1} \frac{u^k}{k^r}
    \right). \end{eqnarray*} Taking $g(u) = \exp
    \left(\sum_{r=2}^a p(a,r) \sum_{k \geq 1} \frac{u^k}{k^r}
    \right)$ proves the result. \end{proof}

{\it Remark} The upper bound in Theorem \ref{genfunction} is not perfect.
In fact when $n=2$, it does not approach 0 as $a \rightarrow \infty$,
whereas the true answer must by Theorem \ref{theC}. However by part 2
of Theorem \ref{genfunction}, the bound is useful for $a$ fixed and $n$
growing, and it will be crucially applied in Section \ref{prim}
when $a=n$ are both growing.

\section{Primitive subgroups} \label{prim}

A main goal of this section is to prove that the proportion of
elements of $S_n$ which belong to a primitive subgroup not containing
$A_n$ is at most $O(n^{-2/3+\alpha})$ for any $\alpha>0$. This
improves on the bound $O(n^{-1/2+\alpha})$ in \cite{LP}, which was
used in proving Theorem \ref{theB} in Section \ref{Onan}.  We conjecture
that  this can in fact be replaced by $O(n^{-1})$ (and the examples
with $n = (q^d-1)/(q-1)$ with the subgroup
containing  $PGL(d,q)$ or $n=p^d$ with subgroup $AGL(d,p)$
show that in general one can do no better).

The minimal degree of a permutation group is defined as the least
number of points moved by a nontrivial element. The first step is to
classify the degree $n$ primitive permutation groups with minimal
degree at most $n^{2/3}$. We note that Babai \cite{babai} gave an
elegant proof (not requiring the classification of finite simple
groups) that there are no primitive permutation groups of degree $n$
other than $A_n$ or $S_n$ with minimal degree at most $n^{1/2}$.

\begin{theorem} \label{bobemailprim} Let $G$ be a primitive permutation
group of degree $n$.
Assume that there is a nontrivial $g \in G$ moving at most $n^{2/3}$ points.
Then one of the following holds:
\begin{enumerate}
\item  $G=A_n$ or $S_n$;
\item  $G=S_m$ or $A_m$ with $m \ge 5$ and
 $n = \binom{m}{2}$ (acting on subsets of
size $2$) ; or
\item  $A_m \times A_m  < G \le S_m \wr S_2$ with $m \ge 4$
and $n=m^2$
(preserving a product structure).
\end{enumerate}
If there is a nontrivial $g \in G$ moving fewer than $n^{1/2}$
points, then $G=A_n$ or $S_n$.
\end{theorem}

\begin{proof} First note that the minimal degree in (2)
is $2(m-2)$ and in (3) is $2n^{1/2}$.  In particular, aside from
(1), we always have the minimal degree is at least $n^{1/2}$. Thus,
the last statement follows from the first part.

It follows by the main result of \cite{GM} that if
there is a $g \in G$ moving fewer than $n/2$ points, then one the
following holds:

(a) $G$ is almost simple with socle (the subgroup generated by the
minimal normal subgroups) $A_m$ and $n = \binom{m}{k}$ with the action
on subsets of size $k < m/2$;

(b) $n=m^t$, with $t > 1$, $m \ge 5$, $G$ has a unique minimal normal
subgroup $N = L_1 \times \ldots \times L_t$ with $t >1$ and $G$
preserves a product structure -- i.e. if $\Omega =\{1, \ldots, n\}$,
then as $G$-sets, $\Omega \cong X^t$ where $m=|X|$, $G \le S_m \wr S_t$
acts on $X^t$ by acting on each coordinate and permuting the
coordinates.

Note that $n/2 \geq n^{2/3}$ as long as $n \geq 8$. If $n<8$, then $G$
contains an element moving at most $3$ points, i.e. either a
transposition or a $3$ cycle, and so contains $A_n$ (Theorem 3.3A in
\cite{DxM}).

Consider (a) above.  If $k=1$, then $(1)$ holds.  If $3 \le k < m/2$,
then it is an easy
exercise to see that the element of $S_m$ moving the fewest $k$ sets is
a transposition.  The number of $k$-sets moved is
$2 \binom{m-2}{k-1}$. We claim that this is greater than $n^{2/3}$.
Indeed, the sought inequality is equivalent to checking that
$\frac{2k(m-k)}{m(m-1)} > {m \choose k}^{-1/3}$.
The worst case is clearly $k=3$, which is checked by
taking cubes. This settles the case $3 \le k < m/2$, and if $k=2$,
we are in case (2).

Now consider (b) above.
Suppose that $t \geq 3$. Then if $g \in S_m
\times \cdots \times S_m$ is nontrivial, it moves at least
$2m^{t-1}>n^{2/3}$ many points. If $g \in S_m \wr S_t$ and is not in
$S_m \times \cdots \times S_m$, then up to conjugacy we may write
$g=(g_1,\cdots,g_t;\sigma)$ where say $\sigma$ has an orbit
$\{ 1, \ldots, s\}$ with $s > 1$.  Viewing
our set as $A \times B$ with $A$ being the first $s$ coordinates, we
see that $g$ fixes at most $m$ points on $A$ (since there is at most one
$g$ fixed point with a given coordinate) and so on the whole space,
$g$ fixes at most $m^{t-s+1} \leq m^{t-1}$ points and so moves at
least $m^t - m^{t-1}$ points. Since $t \ge 3$, this is greater than
$n^{2/3}$. Summarizing, we have shown that in case (b), $t \geq 3$
leads to a contradiction.

So finally consider (b) with $t=2$.  We claim that $L$ must be
$A_m$. Enlarging the group slightly, we may assume that $G = S \wr
S_2$ where $L \le S \le Aut(L)$ and $S$ is primitive of degree $m$.
If $g \notin S \times S$,
 then arguing as in the $t=3$ case shows that
$g$ moves at least $m^2-m$ points. This is greater than $m^{4/3}=n^{2/3}$
since $m \ge 5$, a contradiction. So write
$g = (g_1, g_2) \in S \times S$ with say $g_1 \ne 1$.
If $g_1$ moves at least $d$ points, then $g$ moves at least $dm$ points.
This is greater than $n^{2/3}$ unless $d \le m^{1/3}$.
By this theorem (for $m$),   this implies
that $L=A_m$, whence (1) holds.
\end{proof}

Next, we focus on Case 2 of Theorem \ref{bobemailprim}.

\begin{lemma} \label{cyc} Let $S_m$ be viewed as a subgroup of
$S_{{m \choose 2}}$ using its
 action on 2-sets of $\{1,\cdots,m\}$. For $w \in S_m$, let $A_i(w)$
 denote the number of cycles of $w$ of length $i$ in its usual action
 on $\{1,\cdots,m\}$. The total number of orbits of $w$ on 2-sets
 $\{j,k\}$ of symbols which are in a common cycle of $w$ is \[
 \frac{m}{2} - \sum_{i \ odd} \frac{A_i(w)}{2}.\] \end{lemma}

\begin{proof} First suppose that $w$ is a single cycle of length $i \geq 2$.
If $i$
 is odd, then all orbits of $w$ on pairs of symbols in the $i$-cycle
 have length $i$, so the total number of orbits is $\frac{{i
 \choose 2}}{i} = \frac{i-1}{2}$. If $i$ is even, there is 1 orbit of size
 $\frac{i}{2}$ and all other orbits have size $i$, giving a total of
 $\frac{i}{2}$ orbits. Hence for general $w$, the total number of
 orbits on pairs of symbols in a common cycle of $w$ is \begin{eqnarray*}
\sum_{i \
 odd \atop i \geq 3} A_i(w) \frac{i-1}{2} + \sum_{i \ even} A_i(w)
 \frac{i}{2} & = & \sum_{i \ odd \atop i \geq 1} A_i(w) \frac{i-1}{2} +
 \sum_{i \ even} A_i(w) \frac{i}{2}\\
& = & \frac{m}{2} - \sum_{i \ odd} \frac{A_i(w)}{2}. \end{eqnarray*}
\end{proof}

\begin{theorem} \label{cas2} Let $S_m$ be viewed as a subgroup of
$S_n$ with $n={{m \choose 2}}$ using its action on 2-sets of $\{1,\cdots,m\}$.
Then the proportion of elements of $S_n$ contained in a conjugate
of $S_m$ is at most $O \left( \frac{\log(n)}{n} \right)$.
\end{theorem}

\begin{proof}
 We claim that any element $w$ of $S_m$ has at least $\frac{m}{12}$
 cycles when viewed as an element of $S_n$. Indeed, if $A_1(w)>
 \frac{m}{2}$, then $w$ fixes at least $\frac{m(m-1)}{8} \geq
 \frac{m}{12}$ two-sets. So we suppose that $A_1(w) \leq
 \frac{m}{2}$. Clearly $\sum_{i \geq 3 \  odd} A_i(w) \leq
 \frac{m}{3}$. Thus Lemma \ref{cyc} implies that $w$ has at least
 $\frac{m}{2} - \frac{m}{4} - \frac{m}{6} = \frac{m}{12}$ cycles as an
 element of $S_n$. The number of cycles of a random element of $S_n$
 has mean and variance asymptotic to $\log(n) \sim 2 \log(m)$ (and is
 in fact asymptotically normal) \cite{Go}. Thus by Chebyshev's
 inequality, the proportion of elements in $S_m$ with at least
 $\frac{m}{12}$ cycles is $O \left( \frac{\log(m)}{m^2} \right) = O
 \left( \frac{\log(n)}{n} \right)$, as desired. \end{proof}

To analyze Case 3 of Theorem \ref{bobemailprim}, the following bound,
based on the generating function from
Section \ref{impriv}, will be needed.

\begin{prop} \label{blockbound}  The proportion of elements in $S_{m^2}$
which fix a system of $m$ blocks
of size m is $O(n^{-3/4+\alpha})$ for any $\alpha>0$.
\end{prop}

\begin{proof} By Theorem \ref{genfunction}, the proportion in question is
at most the coefficient of
$u^m$ in \[ \exp \left( \sum_{k \geq 1} \frac{u^k}{mk}
(1+\frac{1}{k}) (1+\frac{1}{2k}) \cdots (1+\frac{1}{(m-1)k}) \right).\]
If $f(u)$ and $g(u)$
are power series in $u$, we write $f(u)<<g(u)$ if the coefficient of $u^n$
in $f(u)$ is less than or equal to
the corresponding coefficient in $g(u)$, for all $n$.
Since $\log(1+x) \leq x$ for $0<x<1$, one has that
\[ \log \left( \prod_{i=1}^{m-1} (1+\frac{1}{ik}) \right) \leq
\sum_{i=1}^{m-1}
\frac{1}{ki} \leq \frac{1}{k} (1+ \log(m-1)).\] Thus \begin{eqnarray*}
& & \exp \left( \sum_{k \geq 1} \frac{u^k}{mk} (1+\frac{1}{k})
(1+\frac{1}{2k}) \cdots (1+\frac{1}{(m-1)k}) \right)\\
& << & \exp
\left( \sum_{k \geq 1} \frac{u^k}{mk} e^{1/k} (m-1)^{1/k} \right)\\ &
<< & e^{ue} \exp \left( \sum_{k \geq 2} \frac{u^k}{k}
\sqrt{\frac{e}{m}} \right)\\ & << & e^{ue} \exp \left( \sum_{k \geq 1}
\frac{u^k}{k} \sqrt{\frac{e}{m}} \right) \\ & = & e^{ue}
(1-u)^{-\sqrt{\frac{e}{m}}}. \end{eqnarray*}

The coefficient of $u^i$ in $(1-u)^{-\sqrt{\frac{e}{m}}}$ is \[
\frac{1}{i!} \sqrt{\frac{e}{m}} \prod_{j=1}^{i-1}
(\sqrt{\frac{e}{m}}+j-1) = \frac{1}{i} \sqrt{\frac{e}{m}}
\prod_{j=1}^{i-1} (1+ \frac{1}{j} \sqrt{\frac{e}{m}}).\] Since \[ \log
\left( \prod_{j=1}^{i-1} (1+\frac{1}{j} \sqrt{\frac{e}{m}}) \right)
\leq \sum_{j=1}^{i-1} \frac{1}{j} \sqrt{\frac{e}{m}} \leq
\sqrt{\frac{e}{m}} (1+ \log(i-1)),\] it follows that \[
\prod_{j=1}^{i-1} (1+ \frac{1}{j} \sqrt{\frac{e}{m}}) \leq
[e(i-1)]^{\sqrt{\frac{e}{m}}}.\] This is at most a universal constant
$A$ if $0 \leq i \leq m$. Thus the coefficient of $u^m$ in $e^{ue}
(1-u)^{-\sqrt{\frac{e}{m}}}$ is at most \[ \frac{e^m}{m!} + A
\sqrt{\frac{e}{m}} \sum_{i=1}^m \frac{1}{i} \frac{e^{m-i}}{(m-i)!}.\]
By Stirling's formula (page 52 of \cite{Fe}), $m!> m^m
e^{-m+1/(12m+1)} \sqrt{2 \pi m}$, which implies that the first term is
very small for large $m$. To bound the sum, consider the terms for
$i \ge  m^{1-\alpha}$, where $0<\alpha<1$. These contribute at most
$\frac{B m^{\alpha}}{m^{3/2}}$ for a universal constant $B$. The
contribution of the other terms is negligible in
comparison, by Stirling's formula. Summarizing, the contribution of
the sum is $O(m^{-3/2+\alpha}) = O(n^{-3/4+\alpha/2})$, as
desired. \end{proof}

The following theorem gives a bound for Case 3 of Theorem \ref{bobemailprim}.

\begin{theorem} \label{cas3} Let $S_m \wr S_2$ be viewed as a
subgroup of $S_n$ with $n=m^2$ using its action on
the Cartesian product $\{1,\cdots,m\}^2$. Then the proportion of elements
of $S_n$ conjugate to
an element of $S_m \wr S_2$ is $O(n^{-3/4+\alpha})$ for any $\alpha>0$.
\end{theorem}

\begin{proof} Consider elements of $S_m \wr S_2$ of the form $(w_1,w_2;id)$.
These all fix $m$ blocks of size $m$ in the action on $\{1,\cdots,m\}^2$;
the blocks consist of points with a given first coordinate. By
Proposition \ref{blockbound}, the proportion of elements of $S_n$
conjugate to some $(w_1,w_2;id)$ is $O(n^{-3/4+\alpha})$ for any
$\alpha>0$.

Next, consider an element of $S_m \wr S_2$ of the form
$\sigma=(w_1,w_2;(12))$. Then $\sigma^2=(w_1w_2,w_2w_1;id)$. Note that
$w_1w_2$ and $w_2w_1$ are conjugate in $S_m$, and let $A_i$ denote
their common number of $i$-cycles. Observe that if $x$ is in an
$i$-cycle of $w_1w_2$, and $y$ is in an $i$-cycle of $w_2w_1$, then
$(x,y) \in \Omega$ is in an orbit of $\sigma^2$ of size $i$. Hence the
total number of orbits of $\sigma^2$ of size $i$ is at least
$\frac{(iA_i)^2}{i} \geq iA_i$. Thus the total number of orbits of
$\sigma^2$ on $\Omega$ is at least $\sum_i iA_i=m$. Hence the total
number of orbits of $\sigma$ is at least $\frac{m}{2}$. Arguing as in
the proof of Theorem \ref{cas2}, it follows that the proportion of
elements of $S_n$ conjugate to an element of the form $\sigma$ is $O
\left( \frac{\log(n)}{n} \right)$, and so is $O(n^{-3/4+\alpha})$ for
any $\alpha>0$. \end{proof}

Now the main result of this section can be proved.

\begin{theorem} \label{mainres} The proportion of elements of
$S_n$ which belong to a
primitive subgroup not containing $A_n$ is at most $O(n^{-2/3+\alpha})$
for any $\alpha>0$.
\end{theorem}

\begin{proof} Fix $\alpha>0$. By Bovey \cite{Bo}, the proportion of
 elements $w$ of $S_n$ such that $\langle w \rangle$
has minimum degree greater than
 $n^{2/3}$ is $O(n^{-2/3+\alpha})$. Thus the proportion of $w \in S_n$
 which lie in a primitive permutation group having minimal degree greater
 than  $n^{2/3}$ is $O(n^{-2/3+\alpha})$. The only primitive
 permutation groups of degree $n$ with minimal degree $\leq n^{2/3}$, and
 not containing $A_n$ are given by Cases 2 and 3 of Theorem
 \ref{bobemailprim}. Theorems \ref{cas2} and \ref{cas3} imply that the
 proportion of $w$ lying in the union of all such subgroups is
 $O(n^{-2/3+\alpha})$, so the result follows. \end{proof}

A trivial corollary to the theorem is that this holds for $A_n$ as well. \\

{\it Remark}\   The actions of the symmetric group studied in this
section embed the group as a subgroup of various larger symmetric
groups. Any such embedding can be thought of as a code in the larger
symmetric group. Such codes may be used for approximating sums of
various functions over the larger symmetric group via a sum over the
smaller symmetric group. Our results can be interpreted as giving
examples of functions where the approximation is not particularly
accurate.
For example, the proof of Theorem \ref{cas2} shows this to be the case
when $S_m$ is viewed as a subgroup of $S_n, n = \binom{m}{2}$
using the actions on
2-sets, and the function is the number of cycles.

%%For example, if the proportion of derangements in
%%$S_m$ acting on the $n = \binom{m}{2}$  two element sets is used to
%%approximate the proportion of derangements in $S-n$, the approximation
%%is off for large $m$/

\section{Related results and applications} \label{survey}

  There are numerous applications of the distribution of fixed points
and derangements. Subsection \ref{motivnum} mentions some motivation
from number theory. Subsection \ref{shalev} discusses some literature
on the proportion of derangements and an analog of the main result of
our paper for finite classical groups. Subsection \ref{fpr} discusses
fixed point ratios, emphasizing the application to random
generation. Subsection \ref{miscell} collects some miscellany about
fixed points and derangements, including algorithmic issues and
appearances in algebraic combinatorics.

  While this section does cover many topics, the survey is by no means
  comprehensive. Some splendid presentations of several other topics
  related to derangements are Serre \cite{Se}, Cameron's lecture notes
  \cite{Ca2} and Section 6.6 of \cite{Ca1}.  For the connections with
permutations
  with restricted positions and rook polynomials see
  \cite[2.3, 2.4]{Stanley}.

\subsection{Motivation from number theory} \label{motivnum}

  We describe two number theoretic applications of derangements which
  can be regarded as motivation for their study:\\

(1)  {\it Zeros of polynomials} Let $h(T)$ be a polynomial with
integer coefficients which is irreducible over the integers.
Let $\pi(x)$ be the number of primes $\leq x$ and let $\pi_h(x)$ be
the number of primes $\leq x$ for which
$h$ has no zeros mod $p$. It follows from Chebotarev's density theorem
(see \cite{LS} for history and a proof sketch), that $lim_{x
\rightarrow \infty} \frac{\pi_h(x)}{\pi(x)}$ is equal to the
proportion of derangements in the Galois group $G$ of $h(T)$ (viewed
as permutation of the roots of $h(T)$). Several detailed examples are
worked out in Serre's survey \cite{Se}.

In addition, there are applications such as the the number field sieve
for factoring integers (Section 9 of \cite{BLP}), where it is
important to understand the proportion of primes for which $h$ has no
zeros mod $p$. This motivated Lenstra (1990) to pose the question of
finding a good lower bound for the proportion of derangements of a
transitive permutation group acting on a set of $n$ letters with $n
\geq 2$. Results on this question are described in Subsection
\ref{shalev}.  \\

(2)  {\it The value problem} Let $\mathbb{F}_q$ be a finite field of
size $q$ with characteristic $p$ and let $f(T)$ be a polynomial of
degree $n>1$ in $\mathbb{F}_q[T]$ which is not a polynomial in
$T^p$. The arithmetic question raised by Chowla \cite{Ch} is to
estimate the number $V_f$ of distinct values taken by $f(T)$ as $T$
runs over $\mathbb{F}_q$.

There is an asymptotic formula for $V_f$ in terms of certain Galois
groups and derangements. More precisely, let $G$ be the Galois group
of $f(T)-t=0$ over $\mathbb{F}_q(t)$ and let $N$ be the Galois group
of $f(T)-t=0$ over $\overline{\mathbb{F}}_q(t)$, where
$\overline{\mathbb{F}}_q$ is an algebraic closure of $\mathbb{F}_q$
(we are viewing $f(T)-t$ as a polynomial with variable $T$ with
coefficients in $\F_q(t)$).  Both groups act transitively on the $n$
roots of $f(T)-t=0$. The geometric monodromy group $N$ is a normal
subgroup of the arithmetic monodromy group $G$. The quotient group
$G/N$ is a cyclic group (possibly trivial).

\begin{theorem} (\cite{Co}) Let $xN$ be the coset which is the Frobenius
generator of the cyclic group $G/N$. The Chebotarev density theorem
for function fields yields the following asymptotic formula:
\[ V_f = \left( 1 - \frac{|S_0|}{|N|} \right) q + O(\sqrt{q}) \]
where $S_0$ is the set of group elements in the coset $xN$ which
act as derangements on the set of roots of $f(T)-t=0$.
The constant in the above error term depends only on $n$, not on $q$.
\end{theorem}

As an example, let $f(T)=T^r$ with $r$ prime and different from $p$
(the characteristic of the base field $\mathbb{F}_q$). The Galois
closure of $\mathbb{F}_q(T)/\mathbb{F}_q(T^r)$ is
$\mathbb{F}_q(\mu,T)$ where $\mu$ is a nontrivial $r$th root of $1$.
Thus $N$ is cyclic of order $r$ and $G/N$ is isomorphic to the
Galois group of $\mathbb{F}_q(\mu)/\mathbb{F}_q$.
  The permutation action is of degree $r$. If $G=N$, then
every non-trivial element is a derangement and so the image of $f$ has
order roughly $\frac{q}{r} + O(\sqrt{q})$. If $G \neq N$, then $G$ is a
Frobenius group and  every fixed point free element is contained in $N$.
Indeed, since in this case $(r,q-1)=1$, we see that $T^r$ is bijective
on $\mathbb{F}_q$.
For further examples, see Guralnick-Wan \cite{GW} and references therein.
Using work on derangements, they prove that if the degree of $f$ is relatively
prime to the characteristic $p$, then either $f$ is bijective or
$V_f \leq \frac{5q}{6} + O(\sqrt{q})$.

\subsection{Proportion of derangements and Shalev's conjecture} \label{shalev}

Let $G$ be a finite permutation group acting transitively on a set $X$
of size $n>1$. Subsection \ref{motivnum} motivated the study of
$\delta(G,X)$, the proportion of derangements of $G$ acting on $X$. We
describe some results on this question, focusing particularly on lower
bounds and analogs of our main results for classical groups.

Perhaps the earliest such result is due to Jordan \cite{Jo}, who
showed that $\delta(G,X)>0$. Cameron and Cohen \cite{CaCo} proved that
$\delta(G,X) \geq 1/n$ with equality if and only if $G$ is a Frobenius
group of order $n(n-1)$, where $n$ is a prime power.   See also \cite{Se}, who
also notes a topological application of Jordan's theorem.

Based on extensive computations, it was asked in \cite{Bet} whether
there is a universal constant $\delta > 0$ (which they speculate may
be optimally chosen as $\frac{2}{7}$) such that $\delta(G,X)> \delta$
for all finite simple groups $G$. The existence of such a $\delta > 0$
was also conjectured by  Shalev.

Shalev's conjecture was proved by Fulman and Guralnick in the series
of papers \cite{FG1},\cite{FG2},\cite{FG3}. We do not attempt to
sketch a proof of Shalev's conjecture here, but make a few remarks:

\begin{enumerate}
\item One can assume that the action of $G$ on $X$ is primitive, for
if $f:Y \mapsto X$ is a surjection of $G$-sets, then $\delta(G,Y) \geq
\delta(G,X)$.

\item By Jordan's theorem \cite{Jo} that $\delta(G,X)>0$, the proof of
Shalev's conjecture is an asymptotic result: we only need to show
that there exists a $\delta>0$ such that for any sequence $G_i,X_i$
with $|X_i| \rightarrow \infty$, one has that $\delta(G_i,X_i)>\delta$
for all sufficiently large $i$.

\item When $G$ is the alternating group, by
Theorem \ref{theC},
 for all primitive actions of $A_n$ except the action on $k$-sets,
the proportion of derangements tends to $1$. For the case of $A_n$ on
$k$-sets, one can give arguments similar to those Dixon \cite{Dx1},
who proved that the proportion of elements of $S_n$ which are
derangements on $k$-sets is at least $\frac{1}{3}$.

\item When $G$ is a finite Chevalley group, the key is to study the
set of regular semisimple elements of $G$. Typically (there are some
exceptions in the orthogonal cases) this is the set of elements of $G$
whose characteristic polynomial is square-free. Now a regular
semisimple element is contained in a unique maximal torus, and there
is a map from maximal tori to conjugacy classes of the Weyl
group. This allows one to relate derangements in $G$ to derangements
in the Weyl group. For example, one concludes that the proportion of
elements of
$GL(n,q)$ which are regular semisimple and fix some k-space is at most the
proportion of elements in
$S_n$ which fix a k-set. For large $q$, algebraic group arguments show that
nearly all elements of $GL(n,q)$
are regular semisimple, and for fixed $q$, one uses generating functions to
uniformly bound the proportion of
regular semisimple elements away from $0$.
\end{enumerate}

To close this subsection, we note that the main result of this paper has an
analog for finite classical groups.
The following result was stated in \cite{FG1} and is proved in \cite{FG2}.

\begin{theorem}  Let $G_i$ be a sequence of classical groups
with the natural module of dimension $d_i$.  Let
$X_i$ be a $G_i$-orbit of either totally singular or nondegenerate
subspaces (of the natural module) of dimension $k_i \le d_i/2$.
If $k_i \rightarrow \infty$, then $\lim \delta(G_i,X_i)=1$.
If $k_i$ is a bounded sequence, then there exist
$0 < \delta_1 <  \delta_2 < 1$
so that $\delta_1 <  \delta(G_i,X_i) < \delta_2$.
\end{theorem}

This result applies to any subgroup between the classical
group and its socle.  Note that in the case that $G_i= PSL$, we
view all subspaces as being totally singular (note that the totally
singular spaces have parabolic subgroups as stabilizers).
We also remark that in characteristic $2$, we consider the orthogonal group
inside the symplectic group as the stabilizer of a subspace (indeed, if
we view $Sp(2m,2^e)=O(2m+1,2^e)$, then the orthogonal groups
are stabilizers of nondegenerate hyperplanes).

In fact, Fulman and Guralnick prove an analog of the Luczak-Pyber result
for symmetric groups.  This result was proved by Shalev \cite{Sh1} for
$PGL(d,q)$ with $q$ fixed.

\begin{theorem}  Let $G_i$ be a sequence of simple classical groups
with the natural module $V_i$ of dimension $d_i$ with $d_i \rightarrow \infty$.
Let $H_i$ be the union of all proper irreducible subgroups
(excluding orthogonal subgroups of the symplectic
group in characteristic $2$).  Then
$\lim_{i \rightarrow \infty} |H_i|/|G_i| = 0$.
\end{theorem}

If the $d_i$ are fixed, then this result is false.  For example,
if  $G_i=PSL(2,q)$  and $H$ is the normalizer of a maximal
torus
of $G$, then $\lim_{q \rightarrow \infty} \delta(G, G/H) =  1/2$.
However, the analog of the previous
theorem is proved in \cite{FG1}  if the rank of the Chevalley group is fixed.
In this case, we take $H_i$ to be the union of maximal subgroups
which do not contain a maximal torus.

The example given above shows that the rank of the permutation action
going to $\infty$ does not
imply that the proportion of derangements tends to $1$.  The results of
Fulman and Guralnick do show this is true if one considers simple Chevalley
groups over fields of bounded size.

\subsection{Fixed point ratios} \label{fpr}

Previous sections of this paper have discussed $fix(x)$, the number of
fixed points of an element $x$ of $G$ on a set $\Omega$.
This subsection concerns the fixed point ratio
$rfix(x)=\frac{fix(x)}{|\Omega|}$. We describe applications to
random generation. For many other applications
(base size, Guralnick-Thompson conjecture, etc.), see the survey \cite{Sh2}.
It should also be mentioned that fixed point ratios are a special case of
character ratios, which have numerous applications to areas such as random
walk \cite{D} and number theory \cite{GlM}.

Let $P(G)$ denote the probability that two random elements of a finite
group $G$ generate $G$. One of the first results concerning $P(G)$ is
due to Dixon \cite{Dx2}, who proved that $lim_{n \rightarrow \infty}
P(A_n) = 1$. The corresponding result for finite simple classical
groups is due to Kantor and Lubotzky \cite{KL}. The strategy adopted
by Kantor and Lubotzky was to first note that for any pair $g,h \in
G$, one has that $\langle g,h \rangle \neq G$ if and
only if $\langle g,h \rangle$ is contained in
a maximal subgroup $M$ of $G$. Since $P(g,h \in M) = (|M|/|G|)^2$, it
follows that \[ 1 - P(G) \leq \sum_M \left( \frac{|M|}{|G|} \right)^2
\leq \sum_i \left( \frac{|M_i|}{|G|} \right)^2 \left(
\frac{|G|}{|M_i|} \right) = \sum_i \frac{|M_i|}{|G|}.\] Here $M$
denotes a maximal subgroup and $\{\it M_i \}$ are representatives
of conjugacy classes of maximal subgroups. Roughly, to show that this
sum is small, one can use Aschbacher's classification of maximal
subgroups \cite{As}, together with Liebeck's upper bounds on sizes of
maximal subgroups \cite{Li}.

Now suppose that one wants to study $P_x(G)$, the chance that a fixed
element $x$ and a random element $g$ of $G$ generate $G$. Then \[ 1
- P_x(G) = P(\langle x,g \rangle \neq G)
\leq \sum_{M \ maximal \atop x \in M} P(g
\in M) = \sum_{M \ maximal \atop x \in M} \frac{|M|}{|G|}.\] Here the
sum is over maximal subgroups $M$ containing $x$. Let $\{M_i\}$ be a
set of representatives of maximal subgroups of $G$, and write $M \sim
M_i$ if $M$ is conjugate to $M_i$. Then the above sum becomes \[
\sum_i \frac{|M_i|}{|G|} \sum_{M \sim M_i \atop x \in M} 1.\] To
proceed further we assume that $G$ is simple. Then, letting $N_G(M_i)$
denote the normalizer of $M_i$ in $G$, one has that \[ g_1 M_i
g_1^{-1} = g_2 M_i g_2^{-1} \leftrightarrow g_1^{-1}g_2 \in N_G(M_i)
\leftrightarrow g_1^{-1}g_2 \in M_i \leftrightarrow g_1 M_i = g_2
M_i.\] In other words, there is a bijection between conjugates of
$M_i$ and left cosets of $M_i$. Moreover, $x \in
gM_ig^{-1}$ if and only if $xgM_i = gM_i$. Thus \[ \frac{|M_i|}{|G|}
\sum_{M \sim M_i \atop x \in M} 1 = rfix(x,M_i). \] Here $rfix(x,M_i)$
denotes the fixed point ratio of $x$ on left cosets of $M_i$, that is
the proportion of left cosets of $M_i$ fixed by $x$. Summarizing,
$P_x(G)$ can be upper bounded in terms of the quantities
$rfix(x,M_i)$. This fact has been usefully applied in quite a few
papers (see \cite{GK}, \cite{FG4} and the references therein, for
example).

\subsection{Miscellany} \label{miscell}

This subsection collects some miscellaneous facts about fixed points and
derangements.

(1) {\it Formulae for fixed points}

We next state a well-known elementary proposition which gives
different formulae for the number of fixed points of an element in a
group action.

\begin{prop} Let $G$ be a finite group acting transitively on $X$.
Let $C$ be a conjugacy class of $G$ and $g$ in $C$. Let $H$ be the
stabilizer of a point in $X$.
\begin{enumerate}
\item The number of fixed points of $g$ on $X$ is $\frac{|C \cap H|}{|C|} |X|$.
\item The fixed point ratio of $g$ on $X$ is $\frac{|C \cap H|}{|C|}$.
\item The number of fixed points of $g$ on $X$ is $|C_G(g)| \sum_i
|C_H(g_i)|^{-1}$ where the $g_i$ are representatives for the
$H$ classes of $C \cap H$.
\end{enumerate}
\end{prop}

\begin{proof} Clearly (1) and (2) are equivalent. To prove (1),
we determine the cardinality of the set
$\{(u,x) \in C \times X | ux=x \}$. On the one hand,
this set has size $|C| f(g)$ where $f(g)$ is the
number of fixed points of $g$. On the other hand,
it is $|X| |C \cap H|$, whence (a) holds.

For (c), note that $|C|= \frac{|G|}{|C_G(g)|}$ and $|C \cap H| =
\sum_i \frac{|H|}{|C_H(g_i)|}$ where the $g_i$ are representatives for
the $H$-classes of $C \cap H$. Plugging this into (1) and using $|X|=
\frac{|G|}{|H|}$ completes the proof. \end{proof}

%%%We point out the analog for algebraic groups.

%%%\begin{prop} Let $G$ be an algebraic group over an algebraically closed
%%%field act transitively on a variety $X$.
%%% Let $H$ be the stabilizer of $x \in X$.
%%%  Let $C$ be a conjugacy class of $G$.  If $g \in G$, let $X(g)$ be
%%%the set of fixed points of $g$ on $X$.  If $X(g)$ is nonempty, then
%%%$$
%%%\dim X(g)  = \dim C \cap H  - \dim C +   \dim X.
%%%$$
%%%\end{prop}

%%%\begin{proof}  W
%%%Let $Y=\{(c,x) \in C \times X | cx=x\}$.  This is a closed subvariety
%%%of $C \times X$ and each projection is surjective.

%%% Consider the second projection.  All fibers
%%%are isomorphic and have dimension $C \cap H$.  Thus,
%%%$\dim Y = \dim C \cap H + \dim X$.   Projecting onto the first factor
%%%shows that $\dim Y = \dim C + \dim X(g)$.   Thus,
%%%$\dim X(g)=\dim C \cap H - \dim C + \dim X$, as required.
%%%\end{proof}

(2)  {\it Algorithmic issues}

It is natural to ask for an algorithm to generate a random
derangement in $S_n$, for example for cryptographic purposes. Of
course, one method is to simply generate random permutations until a
derangements is reached. A more closed form algorithm has been
suggested by Sam Payne. This begins by generating a random
permutation and then, working left to right, each fixed point is
transposed with a randomly chosen place. Each such transposition
decreases the number of fixed points and a clever non-inductive
argument shows that after one pass, the resulting derangement is
uniformly distributed. We do not know if this works starting with
the identity permutation instead of a random permutation.

A very different, direct algorithm for generating a uniformly chosen
derangement appears in \cite{De}. There is also a literature on Gray codes
for running through all derangements in the symmetric group; see \cite{BV} and
\cite{KoL}. \\

%%{\bf Jason will think briefly about what happens when you start at the
%%identity}

(3) {\it Algebraic combinatorics}

The set of derangements has itself been the subject of some
combinatorial study. For example, D\'{e}sarm\'{e}nien \cite{De} has shown that
there is a bijection between derangements in $S_n$ and the set of
permutations with first ascent occurring in an even position. This is
extended and refined by D\'{e}sarm\'{e}nien and Wachs
\cite{DeW}. Diaconis, McGrath, and Pitman \cite{DMP} study the set of
derangements with a single descent. They show that this set admits an
associative, commutative product and unique factorization into cyclic
elements. B\'{o}na \cite{Bn} studies the distribution of cycles in
derangements, using among other things a result of E. Canfield that
the associated generating function has all real zeros.  \\

(4) {\it Statistics}

The fixed points of a permutation give rise to a useful metric on
the permutation group: the Hamming metric. Thus $d(\pi,\sigma)$ is equal
to the number of places where $\pi$ and $\sigma$ disagree. This is a
bi-invariant metric on the permutation group and
$$d(\pi,\sigma) = d(id,\pi^{-1} \sigma) = \mbox{number of fixed points in \ }
\pi^{-1} \sigma.$$ Such metrics have many statistical applications
(Chapter 6 of \cite{D}).

\end{document}